\documentclass[reqno]{amsart}
\usepackage{color}
\usepackage{amssymb,amsmath,amsthm,amstext,amsfonts}
\usepackage[dvips]{graphicx}
\usepackage{psfrag}
\usepackage{url}
\usepackage[pagewise]{lineno}

\makeatletter \@addtoreset{equation}{section} \makeatother

\renewcommand\thetable{\thesection.\@arabic\c@table}

\theoremstyle{plain}
\newtheorem{maintheorem}{Theorem}

\newtheorem{theorem}{Theorem}[section]
\newtheorem{proposition}{Proposition}[section]
\newtheorem{lemma}{Lemma}[section]
\newtheorem{corollary}{Corollary}[section]
\newtheorem{claim}{Claim}[section]
\newtheorem{definition}{Definition}[section]
\newtheorem{remark}{Remark}[section]

\begin{document}

\title{Dimension approximation in smooth dynamical systems}

\author{Yongluo Cao }
\address{Departament of Mathematics, Shanghai Key Laboratory of PMMP, East China Normal University,
 Shanghai 200062, P.R. China}
\address{Departament of Mathematics, Soochow University,
Suzhou 215006, Jiangsu, P.R. China}
\address{Center for  Dynamical Systems and Differential Equations, Soochow University, Suzhou 215006, Jiangsu, P.R. China}
\email{ylcao@suda.edu.cn}

\author{Juan Wang}
\address{School of Mathematics, Physics and Statistics, Shanghai University of Engineering Science, Shanghai 201620, P.R. China}
\email{wangjuanmath@sues.edu.cn}

\author{Yun Zhao}
\address{School of Mathematical Sciences, Soochow University,
  Suzhou 215006, Jiangsu, P.R. China}
  \address{Center for  Dynamical Systems and Differential Equations, Soochow University, Suzhou 215006, Jiangsu, P.R. China}
\email{zhaoyun@suda.edu.cn}


\begin{abstract}For a non-conformal repeller $\Lambda$  of a $C^{1+\alpha}$ map $f$ preserving an ergodic measure $\mu$ of positive entropy, this paper shows that the Lyapunov dimension of $\mu$ can be approximated gradually by the Carath\'{e}odory singular dimension of a sequence of horseshoes. For a $C^{1+\alpha}$ diffeomorphism $f$ preserving a hyperbolic ergodic  measure $\mu$ of positive entropy, if $(f, \mu)$ has only two Lyapunov exponents $\lambda_u(\mu)>0>\lambda_s(\mu)$, then the Hausdorff or lower box or upper box dimension of $\mu$ can be approximated by the corresponding dimension of the horseshoes $\{\Lambda_n\}$. The same statement holds true if $f$ is a $C^1$ diffeomorphism with a dominated Oseledet's splitting with respect to $\mu$.
 \end{abstract}



\maketitle

\section{Introduction }
In smooth dynamical systems, a fundamental approximation result asserts  that a $C^{1+\alpha}$ diffeomorphism $f$ which preserves a hyperbolic ergodic measure $\mu$ of positive  entropy can be approximated gradually by compact invariant locally maximal hyperbolic sets--horseshoes $\{\Lambda_n\}$, in the sense that dynamical quantities on the horseshoes such as the topological entropy and pressure,  Lyapunov
exponents and averages of continuous functions are approaching to the ones of the measure $\mu$.

This type of results are widely referred to the landmark work by Katok \cite{ka80} or Katok and Mendoza (see \cite{km95}). Misiurewicz and Szlenk \cite{ms80} earlier proved a related result for continuous and for piecewise monotone maps of the interval. Przytycki and Urba\'{n}ski \cite{PU10} obtained corresponding properties for holomorphic maps in the case of a measure with only positive Lyapunov exponent. A related setting of dyadic diophantine approximations is established by  Persson and Schmeling in \cite{PS08}. For a general $C^{1+\alpha}$ diffeomorphism $f$ preserving a hyperbolic ergodic measure $\mu$ with positive entropy, assume that $\mu$ has $\ell$ different Lyapunov exponent $\{\lambda_j\}_{j=1}^{\ell}$, on each approaching  horseshoe $\Lambda_n$,
Avila, Crovisier and Wilkinson \cite{acw} obtained a continuous splitting
\[T_{\Lambda_n}M=E_1 \oplus E_2 \oplus \cdots \oplus E_\ell \]
and showed that  the exponential growth of  $D_xf^n|_{E_i}$ is roughly $\lambda_i$  for each $i=1, 2, \cdots, \ell$. A corresponding statement for $C^{1+\alpha}$ non-conformal transformations (i.e., non-invertible maps) was shown in \cite{cpz}. See Chung \cite{Chu99},  Gelfert \cite{Gel10, Gel16} and Yang \cite{Yan15}
for other results related to
Katok's approximation construction  of $C^{1+\alpha}$ maps.

A natural question is how large that part of the dynamics described by these horseshoes is. So, it is interesting to estimate the Hausdorff dimension of the stable and/or unstable Cantor sets of a horseshoe.
If $\mu$ is a SRB measure (i.e. a measure with a particular absolute continuity
property on unstable manifolds; see \cite{bp13} for precise definitions), it was showed in \cite{San02} that $\mu$  can be approximated by ergodic measures supported on horseshoes with arbitrarily large unstable dimensions, which generalized Mendoza's result in \cite{men1988} for diffeomorphisms in higher dimensional manifold.
 The approach in \cite{San02} was based on Markov towers that can be described by horseshoes with infinitely many branches and variable return times.
However, there is an essential mistake in the proof of the key Proposition 5.1 in \cite{San02}.
The authors in \cite{wqc} proved  the same result by a different method.
They used the u-Gibbs property of the conditional measure of the equilibrium measure and the properties
of the uniformly hyperbolic dynamical systems.
Furthermore,  in \cite{wqc} the authors  proved that the Hausdorff dimension of $\mu$ can be approximated gradually by the Hausdorff dimension of the horseshoes $\{\Lambda_n\}$ provided that  the stable direction is one dimension.
See also \cite{LS13, Men85, Men89, MS19, San03,San17} that represent works close to
this topic. 

In this work, our main task is to compare the dimension of the horseshoes $\{\Lambda_n\}$ and the given hyperbolic ergodic measure $\mu$ of a $C^{r}$ ($r\ge 1$) diffeomorphism in a more general setting that $\mu$ may be not a SRB measure.
 For a non-conformal repeller $\Lambda$  of a $C^{1+\alpha}$ map, utilizing the approximation result in \cite{cpz}, we show that the Lyapunov dimension (see \eqref{Ldim} for the definition) of an $f$-invariant ergodic measure $\mu$ supported on $\Lambda$ can be approximated gradually by  the Carath\'{e}odory singular dimension (see \eqref{Cdim} for the definition) of the horseshoes $\{\Lambda_n\}$. For a $C^{1+\alpha}$ diffeomorphism $f$ preserving a hyperbolic ergodic  measure $\mu$ of positive entropy, if $(f, \mu)$ has only two Lyapunov exponents $\lambda_u(\mu)>0>\lambda_s(\mu)$, then the Hausdorff or lower box or upper box dimension of $\mu$ can be approximated by the corresponding dimension of the horseshoes $\{\Lambda_n\}$. The same statement holds true if $f$ is a $C^1$ diffeomorphism with a dominated Oseledec's splitting w.r.t. $\mu$.

We arrange the paper as follows. In Section \ref{dp}, we give some basic notions
and properties about topological and measure theoretic pressures and dimensions of sets and measures.
Statements of our main results will be given in Section \ref{Results}.
In Section \ref{proof}, we will give the detailed proofs of the main results.

\section{Definitions and preliminaries} \label{dp}
In this section, we recall the definitions of topological pressure and various dimensions of subsets and/or of invariant measures.


\subsection{Topological and measure theoretic pressures}
Let $f: X\to X$ be a continuous transformation on a compact metric space $X$ equipped with metric $d$.  A subset $F\subset X$ is called an $(n, \epsilon)-$separated set with respect to $f$, if for any two different points
$x,y\in F$, we have $d_n(x,y):=\max_{0\leq k\leq n-1}d(f^k(x), f^k(y))>\epsilon.$ A sequence of continuous functions $\Phi=\{\phi_n\}_{n\ge 1}$ is called \emph{sub-additive}, if
$$\phi_{m+n}\leq\phi_n+\phi_m\circ f^n,~~\forall n,m\in \mathbb{N}.$$
Furthermore,  a sequence of continuous functions $\Psi=\{\psi_n\}_{n\ge 1}$  is called \emph{super-additive} if $-\Psi=\{-\psi_n\}_{n\ge 1}$ is sub-additive.

        \subsubsection{Topological pressure defined via separated sets}
        Given a  sub-additive potential $\Phi=\{\phi_n\}_{n\ge 1}$ on $X$, put $$P_n(f, \Phi, \epsilon)=\sup\Big\{\sum_{x\in F}e^{\phi_n(x)}| F\ \mbox{is an}\ (n, \epsilon)-\mbox{separated subset of } X\Big\}.$$

  \begin{definition}
  We call the following quantity
  \begin{eqnarray}\label{top-pressure}
 P_{\mathrm{top}}(f,\Phi)=\lim_{\epsilon\to 0}\limsup_{n\to\infty}\frac{1}{n}\log P_n(f,\Phi,\epsilon)
 \end{eqnarray}
 the \emph{sub-additive  topological pressure} of $(f,\Phi)$.
 \end{definition}

\begin{remark}
If $\varPhi=\{\varphi_n\}_{n\geq1}$ is \emph{additive} in the sense that $\varphi_n(x)=\varphi(x)+\varphi(fx)+\cdots+\varphi(f^{n-1}x)\triangleq S_n\varphi(x)$ for some continuous function $\varphi: X\to \mathbb{R}$, we simply denote the topological pressure $P_{\mathrm{top}}(f, \varPhi)$ as $P_{\mathrm{top}}(f, \varphi)$.
\end{remark}


Let $\mathcal{M}_f(X)$ denote the space of all $f-$invariant measures on $X$. For $\mu\in\mathcal{M}_f(X)$, let $h_\mu(f)$ denote the metric entropy of $f$ with respect to $\mu$ (see Walters' book \cite{wal82} for details of metric entropy), and let
$$\mathcal{L}_*(\Phi,\mu)=\lim_{n\to\infty}\frac 1n \int\phi_nd\mu.$$
The existence of the above limit follows from a sub-additive argument.  In \cite{cfh}, the authors proved the following variational principle.

\begin{theorem}\label{subaddvariprin}
Let $f:X\to X$ be a continuous transformation on a compact metric space $X$, and $\Phi=\{\phi_n\}_{n\ge 1}$ a sub-additive potential on $X$, then we have
$$P_{\mathrm{top}}(f,\Phi)=\sup\Big\{h_\mu(f)+\mathcal{L}_*(\Phi,\mu): \mu\in\mathcal{M}_f(X),~\mathcal{L}_*(\Phi,\mu)\neq -\infty\Big\}.$$
\end{theorem}

Though it is unknown whether the variational principle holds for super-additive topological pressure, Cao, Pesin and Zhao gave an alternative definition via variational principle in \cite{cpz}. Given a sequence of super-additive continuous potentials $\Psi=\{\psi_n\}_{n\geq1}$ on a compact dynamical system $(X, f)$, the super-additive topological pressure of $\Psi$ is defined as
\[P_{\mathrm{var}}(f, \Psi) := \sup\Big\{h_\mu(f)+\mathcal{L}_*(\Psi,\mu): \mu\in\mathcal{M}_f(X)\Big\}.\]
where
$$\mathcal{L}_*(\Psi,\mu)=\lim_{n\to\infty}\frac 1n \int\psi_nd\mu=\sup_{n\geq1}\frac 1n \int\psi_nd\mu.$$
The second equality is due to the standard sub-additive argument. 

          \subsubsection{Measure theoretic pressure}
We first follow the approach in \cite{Pes97} to give  the definitions of topological pressures on arbitrary subsets. Given a sub-additive potential  $\Phi=\{\phi_n\}_{n\ge 1}$ on $X$, a subset $Z\subset X$ and $\alpha\in \mathbb{R}$, let
\[
\begin{aligned}
M(Z,\Phi,\alpha,N,\epsilon)=\inf\Big\{&\sum_i
\exp\bigr(-\alpha n_i+\sup_{y\in B_{n_i}(x_i,\epsilon)}\phi_{n_i}(y)\bigr):\\
&\bigcup_{i}B_{n_i}(x_i,\epsilon)\supset Z, \, x_i\in X \text{ and } n_i\ge N \text{ for all } i\Big\}.
\end{aligned}
\]
Since $M(Z,\Phi,\alpha,N,\epsilon)$ is monotonically
increasing with $N$, let
\begin{eqnarray} \label{c-p}
m(Z,\Phi,\alpha,\epsilon)
:=\lim_{N\rightarrow\infty}M(Z,\Phi,\alpha,N,\epsilon).
\end{eqnarray}
We denote the jump-up point of $m(Z,\Phi,\alpha,\epsilon)$ by
\[
P_{Z}(f,\Phi,\epsilon) =\inf \{ \alpha:
m(Z,\Phi,\alpha,\epsilon)=0 \}
=\sup\{ \alpha: m(Z,\Phi,\alpha,\epsilon)=+\infty\}.
\]

\begin{definition}\label{defPmu*} We call the quantity
\begin{eqnarray*}
P_{Z}(f,\Phi)=\liminf_{\epsilon\to 0} P_{Z}(f,\Phi,\epsilon)
\end{eqnarray*}
the \emph{topological pressure} of $(f, \Phi)$ on the set $Z$ (see \cite{fh16} for the weighted version of this quantity).
\end{definition}

Similarly, for $\alpha\in \mathbb{R}$ and $Z\subset X$, define
\[
R(Z,\Phi,\alpha,N,\epsilon)=\inf\Big\{\sum_i
\exp\bigr(-\alpha N+\sup_{y\in B_{N}(x_i,\epsilon)} \phi_N(y) \bigr):\\
\bigcup_{i}B_{N}(x_i,\epsilon)\supset Z,\,x_i\in X\Big\}.
\]
 We set
$$
\begin{aligned}
\underline{r}(Z,\Phi,\alpha,\epsilon)
&=\liminf_{N\rightarrow\infty}R(Z,\Phi,\alpha,N,\epsilon), \\
\overline{r}(Z,\Phi,\alpha,\epsilon)
&=\limsup_{N\rightarrow\infty}R(Z,\Phi,\alpha,N,\epsilon)
\end{aligned}
$$
and define the jump-up points of
$\underline{r}(Z,\Phi,\alpha,\epsilon)$ and
$\overline{r}(Z,\Phi,\alpha,\epsilon)$ as
$$
\begin{aligned}
\underline{CP}_{Z}(f,\Phi,\epsilon)&=\inf\{\alpha:
\underline{r}(Z,\Phi,\alpha,\epsilon)=0\}
=\sup\{ \alpha: \underline{r}(Z,\Phi,\alpha,\epsilon)=+\infty \},\\
\overline{CP}_{Z}(f,\Phi,\epsilon)&=\inf \{ \alpha:
\overline{r}(Z,\Phi,\alpha,\epsilon)=0 \}
=\sup\{ \alpha: \overline{r}(Z,\Phi,\alpha,\epsilon)=+\infty\}
\end{aligned}
$$
respectively.

\begin{definition}
We call the quantities
$$
\underline{CP}_{Z}(f,\Phi)=\liminf_{\epsilon\to 0}
\underline{CP}_{Z}(f,\Phi,\epsilon)\,\,\text{and}\,\,
\overline{CP}_{Z}(f,\Phi)=\liminf_{\epsilon\to 0}
\overline{CP}_{Z}(f,\Phi,\epsilon)
$$
the \emph{lower} and \emph{upper topological pressures} of $(f,\Phi)$ on the set $Z$ respectively.
\end{definition}

 Given an $f$-invariant measure
$\mu$, let
$$
\begin{aligned}
P_{\mu}(f,\Phi,\epsilon)
=\inf\{P_Z(f,\Phi,\epsilon)\colon\mu (Z)=1\}
\end{aligned}
$$
and then we call the following quantity
\begin{equation*}\label{pressure1}
P_{\mu}(f,\Phi):=\liminf_{\epsilon\to 0}P_{\mu}(f,\Phi,\epsilon)
\end{equation*}
the \emph{measure theoretic pressure} of $(f,\Phi)$ with respect to $\mu$.
 Let further
\begin{gather*}
\underline{CP}_{\mu}(f,\Phi,\epsilon)
=\lim_{\delta\to 0}\inf\{\underline{CP}_{Z}(f,\Phi,\epsilon)
\colon\mu (Z)\ge 1-\delta\}, \\
\overline{CP}_{\mu}(f,\Phi,\epsilon)
=\lim_{\delta\to 0}\inf\{\overline{CP}_{Z}(f,\Phi,\epsilon)
\colon\mu (Z)\ge 1-\delta\}.
\end{gather*}
We call the following quantities
\begin{eqnarray*}
\underline{CP}_{\mu}(f,\Phi)=\liminf_{\epsilon\to 0}
\underline{CP}_{\mu}(f,\Phi,\epsilon),\ \ \
\overline{CP}_{\mu}(f,\Phi)=\liminf_{\epsilon\to 0}
\overline{CP}_{\mu}(f,\Phi,\epsilon)
\end{eqnarray*}
the \emph{lower and upper measure theoretic pressures} of $(f,\Phi)$ with respect to  $\mu$ respectively.
It is proved in \cite[Theorem A]{chz} that
\begin{eqnarray}\label{metric-ent-form}
P_{\mu}(f,\Phi)=\underline{CP}_{\mu}(f,\Phi)=\overline{CP}_{\mu}(f,\Phi)=h_{\mu}(f)+ \mathcal{L}_*(\Phi,\mu)
\end{eqnarray}
for any $f$-invariant ergodic measure $\mu$ with $\mathcal{L}_*(\Phi,\mu)\neq-\infty$.

\begin{remark}In fact, one can show that
\[
\mathcal{P}_\mu(f,\Phi)=\inf\{\mathcal{P}_Z(f,\Phi):\mu(Z)=1\}
\]
here $\mathcal{P}$ denotes either $P$ or $\underline{CP}$ or $\overline{CP}$, see \cite{zhao17} for a proof.
\end{remark}

\subsection{Dimensions of sets and measures}
Now we recall the definitions of Hausdorff and box dimensions of subsets and measures. Given a subset $Z\subset X$, For any $s\ge 0$, let
\[
\mathcal{H}_{\delta}^{s}(Z)=\inf\Big\{\sum\limits_{i=1}^{\infty}(\mbox{diam} U_{i})^{s}:\{U_{i}\bigr\}_{i\ge1}\
\mbox{is\ a cover\ of } Z \ \mbox{with }  \mbox{diam}U_{i}\le \delta, \forall i\ge1 \Big\}
\]
and
\begin{eqnarray*}
\mathcal{H}^{s}(Z)=\lim\limits_{\delta\rightarrow
0}\mathcal{H}_{\delta}^{s}(Z).
\end{eqnarray*}
The above limit exists, though the limit may be infinity. We call $\mathcal{H}^{s}(Z)$ the $s-$dimensional Hausdorff measure of $Z$.

\begin{definition}\label{dim}  The following jump-up value of $\mathcal{H}^{s}(Z)$
$$\dim_{H}Z=\inf\{s:\mathcal{H}^{s}(Z)=0\}=\sup\{s:\mathcal{H}^{s}(Z)=\infty\}$$
is called the \emph{Hausdorff
dimension} of $Z$.
The \emph{lower and upper box dimension} of $Z$ are defined respectively
by
$$\underline{\dim}_BZ=\liminf\limits_{\delta\to0}\frac{\log N(Z,\delta)}{-\log\delta}\ \text{and}\ \overline{\dim}_BZ=\limsup\limits_{\delta\to0}\frac{\log N(Z,\delta)}{-\log\delta},$$
where $N(Z,\delta)$ denotes the least number of balls of radius $\delta$ that are needed to cover the set $Z$. If $\underline{\dim}_BZ=\overline{\dim}_BZ$, we will denote the common value by $\dim_BZ$ and call it the \emph{box dimension} of $Z$.
\end{definition}

The following two results are well-known in the field of fractal geometry, e.g., see Falconer's book \cite{fal03}  for proofs.

\begin{lemma}\label{holderdimension}
Let $X$ and $Y$ be metric spaces. For any $r\in(0, 1)$, $\Phi: X \to Y$ is an onto, $(C, r)$-H$\ddot{o}$lder continuous map for some $C>0$. Then
\[\dim_H Y \leq r^{-1} \dim_H X,\quad \underline{\dim}_B Y \leq r^{-1} \underline{\dim}_B X \quad \text{and} \quad \overline{\dim}_B Y \leq r^{-1} \overline{\dim}_B X.\]
\end{lemma}

\begin{corollary}\label{lipdimension}
Let $X$ and $Y$ be metric spaces, and let $\Phi: X \to Y$ be an onto, Lipschitz continuous map. Then
\[\dim_H Y \leq \dim_H X,\quad \underline{\dim}_B Y \leq \underline{\dim}_B X \quad \text{and} \quad \overline{\dim}_B Y \leq \overline{\dim}_B X.\]
\end{corollary}

Given a Borel probability measure $\mu$ on $X$, the following quantity
\[
\begin{aligned}
\dim_H\mu &=\inf\{\dim_HZ: Z\subset X~~\text{and}~~\mu(Z)=1\}\\
&=\lim_{\delta\to 0} \inf\{\dim_HZ: Z\subset X~~\text{and}~~\mu(Z)\ge1-\delta\}
\end{aligned}
\]
is called the \emph{Hausdorff dimension of the measure} $\mu$. Similarly, we call the following two quantities
\[
\underline{\dim}_B\mu=\lim_{\delta\to 0} \inf\{\underline{\dim}_BZ: Z\subset X~~\text{and}~~\mu(Z)\ge1-\delta\}
\]
and
\[
\overline{\dim}_B\mu=\lim_{\delta\to 0} \inf\{\overline{\dim}_BZ: Z\subset X~~\text{and}~~\mu(Z)\ge1-\delta\}
\]
the \emph{lower box dimension} and \emph{upper box dimension} of $\mu$, respectively.

If $\mu$ is a finite measure on $X$ and there exists $d\geq 0$ such that
\[\lim_{r\to0}\frac{\log\mu(B(x, r))}{\log r} = d\]
for $\mu$-almost every $x\in X$, then
\[\dim_H\mu = \underline{\dim}_B\mu = \overline{\dim}_B \mu=d. \]
This  criterion was established by Young in \cite{young}.

\section{Statements of main results}\label{Results}
In this section, we will give the statements of the main results in this paper, and the proof will be postponed to the next section.
\subsection{Dimension approximation for uniformly expanding systems}
Let $f:M\to M$ be a smooth map of a $m_0$-dimensional compact smooth Riemannian manifold $M$, and $\Lambda$ a compact $f$-invariant subset of $M$. Let $\mathcal{M}_f(\Lambda)$ and $\mathcal{E}_f(\Lambda)$ denote respectively the set of all $f$-invariant measures and ergodic measures on $\Lambda$.

\subsubsection{Definitions of repeller and Lyapunov dimension}
We call $\Lambda$ a \emph{repeller} for $f$ or $f$ is \emph{expanding} on $\Lambda$ if
\begin{enumerate}
\item[(1)]there exists an open neighborhood $U$ of $\Lambda$ such that
$\Lambda=\{x\in U: f^n(x)\in U\,\, \text{for all}\,\, n\ge 0\}$;
\item[(2)]there is $\kappa > 1 $ such that
$$
\|D_xf (v) \| \ge \kappa \|v\|, \, \text{ for all } x \in \Lambda,\, \text{ and } v \in T_xM,
$$
where $\|\cdot\|$ is the norm induced by the Riemannian metric on $M$, and $D_xf:T_xM\rightarrow T_{f(x)}M$ is the differential operator.
\end{enumerate}


 Given an $f$-invariant ergodic measure $\mu$ supported on the repeller $\Lambda$. Let $\lambda_1(\mu)\ge \lambda_2(\mu)\ge \cdots \ge \lambda_{m_0}(\mu)$ and $h_{\mu}(f)$ denote the Lyapunov exponents and the measure theoretic entropy of $(f,\mu)$ respectively, we refer the reader to \cite{bp13} and \cite{wal82} for detailed description of Lyapunov exponents and the the measure theoretic entropy.  We further define the \emph{Lyapunov dimension} of $\mu$ as follows:
 \begin{equation}\label{Ldim}
 \dim_{\mathrm{L}}\mu:=\left \{
 \begin{array}{ll}
 \ell +\frac{h_{\mu}(f)-\lambda_{m_0}(\mu)-\cdots-\lambda_{m_0-\ell +1}(\mu)}{\lambda_{m_0-\ell}(\mu)}, & h_\mu(f)\ge \lambda_{m_0}(\mu) \\
 \frac{h_\mu(f)}{\lambda_{m_0}(\mu)}, & 0\le h_\mu(f)< \lambda_{m_0}(\mu)
 \end{array}
 \right.
 \end{equation}
 where $\ell=\max\{i: \lambda_{m_0}(\mu)+\cdots+\lambda_{m_0-i +1}(\mu)\le h_{\mu}(f)\}$.

 The original definition of Lyapunov dimension in \cite{AY,KY,L81} is defined only for hyperbolic systems as follows: assume that $\nu$ is an ergodic measure of a smooth diffeomorphism $f$  with Lyapunov exponents $\lambda_1\ge\cdots\ge \lambda_u>0\ge \lambda_{u+1}\ge\cdots\ge \lambda_{m_0}$,  then the
the Lyapunov dimension is
 \[
 \mathrm{Lya}\dim\,\nu=\ell +\frac{\lambda_1+\cdots+\lambda_u+\cdots+\lambda_{\ell}}{|\lambda_{\ell+1}|}
 \]
where $\ell=\max\{i:\lambda_1+\cdots+\lambda_i\ge 0\}$. Assume further that $\nu$ is a SRB measure, then $h_\nu(f)=\lambda_1+\cdots+\lambda_u$. In consequence,
\[
\mathrm{Lya}\dim\,\nu=\ell +\frac{h_\nu(f)+\lambda_{u+1}+\cdots+\lambda_{\ell}}{|\lambda_{\ell+1}|}
\]
and $\ell=\max\{i:-\lambda_{u+1}-\cdots-\lambda_i\ge h_\nu(f)\}$. Hence, the definition in \eqref{Ldim} is a reasonable substitute. For a $C^1$ expanding map $f$,  Feng and Simon \cite{FK20} defined the Lyapunov dimension of an ergodic measure as the zero of the measure theoretic pressure $P_\mu(f,\Phi_f(t))=0$ (see \eqref{subadditive potentials}). In this paper, we will prove the unique solution of the equation $P_\mu(f,\Phi_f(t))=0$ is indeed our definition of Lyapunov dimension (see Theorem \ref{A}).  Furthermore, this paper shows that the Lyapunov dimension of an ergodic measure  defined in \eqref{Ldim} equals to its Carath\'{e}odory singular dimension (see Proposition \ref{dimequal}), so  the Carath\'{e}odory singular dimension (see Section \ref{csd} for the detailed definition) can be regarded as a geometric explanation of the Lyapunov dimension.

\subsubsection{Singular valued potentials}\label{svp} Let $\Lambda$ be a repeller of a smooth map $f:M\to M$.
Given $x\in\Lambda$ and $n\ge 1$, consider the differentiable operator $D_xf^n: T_xM\to T_{f^n(x)}M$ and  denote the singular values of $D_xf^n$ (square roots of the eigenvalues of $(D_xf^n)^*D_xf^n$) in the decreasing order by
\begin{equation}\label{sing-val}
\alpha_1(x,f^n)\ge\alpha_2(x,f^n)\ge\dots\ge\alpha_{m_0}(x,f^n).
\end{equation}
For $t\in [0,m_0]$, set
\begin{equation}\label{sing-val-phi}
\varphi^{t}(x,f^n):
=\sum_{i=m_0-[t]+1}^{m_0}\log\alpha_i(x,f^n)
+(t-[t])\log\alpha_{m_0-[t]}(x,f^n).
\end{equation}
Since $f$ is smooth, the functions $x\mapsto\alpha_i(x,f^n)$, $x\mapsto\varphi^t(x,f^n)$ are continuous for any $n\ge 1$. It is easy to see that for all $n,\ell\in\mathbb{N}$
$$
\varphi^t(x,f^{n+\ell})\ge\varphi^t(x,f^n)+\varphi^t(f^n(x),f^\ell).
$$
It follows that the sequence of functions
\begin{equation}\label{subadditive potentials}
\Phi_f(t):=\{-\varphi^t(\cdot,f^n)\}_{n\ge 1}
\end{equation}
is sub-additive, which is called the \emph{sub-additive singular valued potentials}.

\subsubsection{Carath\'{e}odory singular dimension}\label{csd}
We recall the definition of Carath\'{e}odory singular dimension of a repeller which is introduced in \cite{cpz}.

Let $\Phi_f(t)=\{-\varphi^{t}(\cdot,f^n)\}_{n\ge 1}$. Given a subset $Z\subseteq \Lambda$, for each  small number $r>0$,  let
\[
m(Z,t,r):=\lim_{N\to\infty}\inf\Big\{\sum_i\exp\bigr(\sup_{y\in B_{n_i}(x_i,r)}-\varphi^{t}(y,f^{n_i})\bigr)\Big\},
\]
where the infimum is taken over all collections $\{B_{n_i}(x_i,r)\}$ of Bowen's balls with $x_i\in\Lambda$, $n_i\ge N$ that cover $Z$. It is easy to see that there is a jump-up value
\begin{equation}\label{dimC}
\dim_{C,r}Z:=\inf\{t: m(Z,t,r)=0\}=\sup\{t: m(Z,t,r)=+\infty\}.
\end{equation}
The following quantity \begin{equation}\label{Cdim}
\displaystyle{\dim_{C}Z:=\liminf_{r\to 0}\dim_{C,r}Z }\end{equation} is called the \emph{Carath\'{e}odory singular dimension of $Z$}. Particularly, the Carath\'{e}odory singular dimension of the repeller $\Lambda$ is independent of the parameter $r$ for small values of $r>0$ (see \cite[Theorem 4.1]{cpz}).

For each $f$-invariant measure $\mu$ supported on $\Lambda$, let
$$
\dim_{C,r}\mu:=\inf\{\dim_{C,r}Z:  \mu(Z)=1 \},
$$
and the following quantity
\[
 \dim_{C}\mu:=\liminf_{r\to 0}\dim_{C,r}\mu
\]
is called the \emph{Carath\'{e}odory singular dimension}  of the measure $\mu$.

\subsubsection{Approximation of Carath\'{e}odory singular dimension of repellers}
Given a repeller $\Lambda$ of a $C^{1+\alpha}$ map $f$, the following result shows that the zero of the measure theoretic pressure function is exactly the Lyapunov dimension of an ergodic measure $\mu\in \mathcal{E}_f(\Lambda)$, and the Lyapunov dimension of an ergodic measure of positive entropy can be approximated by the Carath\'{e}odory singular dimension of a sequence of invariant sets.
 Recall that $
\Phi_f(t):=\{-\varphi^t(\cdot,f^n)\}_{n\ge 1}
$
is the sub-additive singular valued potentials with respect to $f$ (See the definition in (\ref{subadditive potentials}).).

\begin{maintheorem}\label{A}
Let $f: M\to M$ be a $C^{1+\alpha}$ map of an $m_0$-dimensional compact smooth Riemannian manifold $M$, and $\Lambda$ a repeller of $f$. Then the following statements hold:
\begin{enumerate}
\item[(1)] for every $f$-invariant ergodic measure $\mu$ supported on $\Lambda$, we have that $$\dim_{\mathrm{L}}\mu=s_{\mu}$$
     where $s_{\mu}$ is the unique root of the equation $P_\mu(f, \Phi_f(t))=0$;
 \item[(2)] let $\mu$ be an $f$-invariant ergodic measure on $\Lambda$ with $h_{\mu}(f)>0$, for any $\varepsilon>0$ there exists an $f$-invariant compact subset $\Lambda_\varepsilon\subset\Lambda$ such that $\dim_C \Lambda_\varepsilon \rightarrow \dim_{\mathrm{L}}\mu$ as $\varepsilon$ approaching zero.
\end{enumerate}
\end{maintheorem}

Some comments on the previous theorem are in order. First, we would like to point out that it is enough to require $f$ to be a $C^1$ map in  the first
statement, however, the map of higher smoothness
$C^{1+\alpha}$ is crucial in the last statement as it allows us to utilize some powerful results of Pesin theory. Second, if $f$ is a local diffeomorphism preserving  an ergodic expanding measure $\mu$ of positive entropy, i.e., $(f,\mu)$ has only positive Lyapunov exponent, in this case one can also obtain a approximation result as in \cite{cpz} so that we can obtain the second statement in the previous theorem in this setting.  In \cite{San03}, for a $C^2$ interval map $f$ with
finitely many non-degenerate critical points, the author proved that the Hausdorff dimension of an expanding measure $\mu$ can be approximated gradually by the Hausdorff dimension of a sequence of repellers.

For each $f$-invariant ergodic measure $\mu$ supported on $\Lambda$, the following result shows that the Carath\'{e}odory singular dimension of $\mu$ is exactly its Lyapunov dimension.

\begin{proposition}\label{dimequal}Let $f: M\to M$ be a $C^{1}$ map of an $m_0$-dimensional compact smooth Riemannian manifold $M$, and $\Lambda$ a repeller for $f$. Then the following statements hold:
\begin{enumerate}
\item[(1)] for each subset $Z\subset \Lambda$, we have that
\[
\dim_C Z=t_Z
\]
where $t_Z$ is the unique root of the equation $P_Z(f, \Phi_f(t))=0$;
\item[(2)]for each $f$-invariant ergodic measure $\mu$ supported on $\Lambda$, we have that
\[
\dim_C\mu=\dim_{\mathrm{L}}\mu.
\]
\end{enumerate}
\end{proposition}

\subsection{Dimension approximation in non-uniformly hyperbolic systems}In this section, we first recall an approximation result  in non-uniformly hyperbolic systems that are proved by Avila {\it et al} \cite{acw}, then we give the statement of our dimension approximation result in non-uniformly hyperbolic systems.

\subsubsection{Lyapunov exponents and holonomy maps}Let $f: M\to M$ be a  diffeomorphism  on an $m_0$-dimensional compact smooth  Riemannian manifold $M$. By the Oseledec's multiplicative ergodic theorem (see \cite{ose}), there exists a total measure set $\mathcal{O}\subset M$ such that, for each $x \in \mathcal{O}$ and each invariant measure $\mu$ there exist positive integers  $d_1(x),d_2(x),\cdots, d_{p(x)}(x)$, numbers $\lambda_1(x)>\lambda_2(x)>\cdots >\lambda_{p(x)}(x)$ and a splitting $$T_xM=E_1(x)\oplus E_2(x)\oplus\cdots \oplus E_{p(x)}(x)$$ satisfy  that
\begin{enumerate}
\item[(1)] $D_xf E_i(x)=E_i(f(x))$ for each $i$ and $\sum_{i=1}^{p(x)}d_i(x)=m_0$;
\item[(2)] for each $0\neq v\in E_i(x)$ we have that $$\lambda_i(x) = \lim_{n \to \infty}\frac{1}{n}\log\|D_xf^n(v)\|.$$
\end{enumerate}
Here we call the numbers $\{\lambda_i(x)\}_{i=1}^{p(x)}$ the Lyapunov exponents of $(f,\mu)$.
In the case that $\mu$ is  an $f$-invariant ergodic measure, the numbers $p(x)$, $\{d_i(x)\}$ and $\{\lambda_i(x)\}$ are constants almost everywhere. We denote them simply by $p$,  $\{d_i\}_{i=1}^p$ and $\{\lambda_i\}_{i=1}^p$.

A compact invariant subset $\Lambda\subset M$ is called a  \emph{ hyperbolic set}, if there exists a continuous splitting of the tangent bundle $T_\Lambda M = E^{s}\oplus E^{u}$, and constants $C > 0,\ 0 < \lambda < 1$ such that for every $x \in \Lambda$
 \begin{enumerate}
\item[(1)] $D_xf(E^s(x)) = E^s(f(x)),\ D_xf(E^u(x)) = E^u(f(x))$;
\item[(2)] for all $n \geq 0,\ \|D_xf^n(v)\|\leq C\lambda ^n\|v\|$ if $v \in E^s(x)$, and $\|D_xf^{-n}(v)\|\leq C\lambda^n\|v\|$ if $v \in E^u(x)$.
\end{enumerate} Given a point $x \in \Lambda$, for each small $\beta>0$,  the \emph{local stable and unstable manifolds} are defined as  follows:
\begin{eqnarray*}
&&W_\beta^s(f,x)=\Big\{y \in M : d(f^n(x),f^n(y)) \leq \beta,\ \forall n \geq 0\Big\},\\
&&W_\beta^u(f,x)=\Big\{y \in M : d(f^{-n}(x),f^{-n}(y)) \leq \beta,\ \forall n \geq 0\Big\}.
\end{eqnarray*}
The global stable and unstable sets of $x\in\Lambda$ are given as follows:
$$W^s(f,x)=\bigcup_{n\geq0}f^{-n}(W^s_\beta(f,f^n(x))),\, W^u(f,x)=\bigcup_{n\geq0}f^n(W^u_\beta(f,f^{-n}(x))).$$ A hyperbolic set is called \emph{locally maximal}, if there exists a neighbourhood $U$ of $\Lambda$ such that $\Lambda=\bigcap_{n\in\mathbb{Z}}f^n(U)$. Recall that a \emph{horseshoe} for a diffeomorphism $f$ is a transitive, locally maximal hyperbolic set that is totally disconnected and not finite.

Let $W^u$ and $W^s$ be the unstable and stable foliations of hyperbolic dynamical system $(f, \Lambda)$. For $x, y\in \Lambda$ with $x$ close to $y$, let $W_\beta^u(f, x)$ and $W_\beta^s(f, x)$ be the local stable foliations of $x$ and $y$. Define the map $h: W_\beta^s(f, x) \to W_\beta^s(f, y)$
sending $z$ to $h(z)$ by sliding along the leaves of $W^u$. The map $h$ is called the
holonomy map of $W^u$. The map $h$ is Lipschitz continuous if
$$d_y(h(z_1), h(z_2)) \leq L d_x(z_1, z_2),$$
where $z_1, z_2 \in W_\beta^s(f, x)$ and $d_x, d_y$ are natural path metrics on $W_\beta^s(f, x)$, $W_\beta^u(f, y)$ with respect to a fixed Riemannian structure on $M$. The constant $L$ is the Lipschitz constant, and it is independent of the choice of $W^s$. The map $h$ is $\alpha$-H$\ddot{o}$lder continuous if
$$d_y(h(z_1), h(z_2)) \leq H d_x(z_1, z_2)^\alpha,$$
where $H$ is the H\"{o}lder constant. Similarly we can define the holonomy map of $W^s$.

\subsubsection{Approximation of Lyapunov exponents and entropy}
For a $C^{1+\alpha}$ diffeomorphism $f:M\rightarrow M$,  Katok \cite{ka80} showed that an $f$-invariant ergodic hyperbolic measure (a measure has no zero Lyapunov exponents) with positive metric entropy can be approximated by horseshoes. However, Katok's result does not explicitly
mention a control of the Oseledets splitting over the horseshoes.  Recently, Avila {\it et al} \cite{acw} showed that there is a dominated splitting over the horseshoes, with approximately the same Lyapunov exponents on
each sub-bundle of the splitting.

Recall that $Df$-invariant splitting on a compact $f$-invariant subset $\Lambda$
\[
T_\Lambda M=E_1\oplus E_2\oplus \cdots \oplus E_{\ell},~~(\ell\ge 2)
\]
is a \emph{dominated splitting}, if there exists $N\ge 1$ such that for every $x\in\Lambda$, any unit vectors $v,w\in T_xM$:
\[
v\in E_i(x), w\in E_j(x)~~\text{with}~~i<j\Longrightarrow \|D_xf^N(v)\|\ge 2\|D_xf^N(w)\|.
\]
We write $E_1\succeq E_2 \succeq \cdots \succeq E_{\ell}$.
Furthermore, if there are numbers $\lambda_1>\lambda_2>\cdots>\lambda_{\ell}$, constants $C>0$ and $\displaystyle{0<\varepsilon< \min_{1\leq i < \ell} \frac{\lambda_{i}-\lambda_{i+1}}{100}}$ such that for every $x\in\Lambda$, $n\in\mathbb{N}$, $1\leq j \leq \ell$ and each unit vector $u\in E_j(x)$, it holds that
\[C^{-1} e^{n(\lambda_j-\varepsilon)} \leq \|D_xf^n(u)\| \leq C e^{n(\lambda_j+\varepsilon)}, \]
then we say that
\[
T_\Lambda M=E_1\oplus E_2\oplus \cdots \oplus E_{\ell},~~(\ell\ge 2)
\]
is a $\{\lambda_j\}_{1\leq j \leq \ell}-$\emph{dominated splitting}.

For the reader's convenience, we recall Avila, Crovisier and Wilkinson's approximation results in the following:

\begin{theorem}\label{horse-app} Let $f:M\to M$ be a $C^{1+\alpha}$ diffeomorphism,  and $\mu$ an $f$-invariant ergodic hyperbolic measure with $h_\mu(f)>0$. For each $\varepsilon>0$ and a weak-$*$ neighborhood $\mathcal{V}$ of $\mu$ in the space of $f$-invariant probability measures on $M$. Then there exists a compact set $\Lambda_\varepsilon^*\subset M$ and a positive integer $N$ such that the following properties hold:
\begin{enumerate}
\item[(1)] $\Lambda_\varepsilon^*$ is a locally maximal hyperbolic set and topologically mixing with respect
to $f^N$;
\item[(2)] $h_{\mu}(f)-\varepsilon<h_{top}(f,\Lambda_\varepsilon)<h_{\mu}(f)+\varepsilon$ where $\Lambda_\varepsilon=\Lambda_\varepsilon^* \cup f(\Lambda_\varepsilon^*) \cup \cdots f^{N-1}(\Lambda_\varepsilon^*)$;
\item[(3)] $\Lambda_\varepsilon$ is $\varepsilon$-close to the support of $\mu$ in the Hausdorff distance;
\item[(4)] each invariant probability measure supported on the horseshoe $\Lambda_\varepsilon$ lies in $\mathcal{V}$;
\item[(5)] if $\lambda_1>\lambda_2>\cdots>\lambda_{\ell}$ are the distinct Lyapunov exponents of $(f,\mu)$, with multiplicities $d_1, d_2,\cdots, d_{\ell}$, then there exists a $\{\lambda_j\}_{1\le j<\ell }-$dominated splitting $T_{\Lambda_\varepsilon} M=E_1\oplus E_2\oplus \cdots \oplus E_{\ell}$ with $\dim E_i=d_i$ for each $i$, and for each $x\in \Lambda_\varepsilon$, $k\geq 1$ and each vector $v\in E_i(x)$
     \[
     e^{(\lambda_i-\varepsilon)kN}\leq \|D_xf^{kN}(v)\|\leq e^{(\lambda_i+\varepsilon)kN},~~~\forall i=1,2,\cdots,\ell.
     \]
\end{enumerate}
\end{theorem}

\begin{remark} In the second statement, the original result does not give the inequality of the right hand side. However, only a slightly modification can give the upper bound of the topological entropy of $f$ on the horseshoe.
\end{remark}

\subsubsection{Statements of results}
Let $f: M\to M$ be a $C^{1+\alpha}$ diffeomorphism of a compact Riemannian manifold $M$, and let $\mu$ be a hyperbolic ergodic $f$-invariant probability measure with positive  entropy. Suppose that $(f,\mu)$ has only two Lyapunov exponents $\lambda_u(\mu)>0>\lambda_s(\mu)$. Ledrappier, Young \cite{ly1} and Barreira, Pesin, Schmeling \cite{bps} proved that
\begin{equation}\label{LY1+}
 \mathrm{Dim}\mu = \frac{h_\mu(f)}{\lambda_u(\mu)} - \frac{h_\mu(f)}{\lambda_s(\mu)}
\end{equation}
where $\mathrm{ Dim}$ denotes either $\dim_H$ or $\underline{\dim}_B$ or $\overline{\dim}_B$.
Our strategy used to prove the dimension approximation in this setting is as follows.
It follows from Theorem \ref{horse-app}  that $\mu$ can be approximated by a sequence of horseshoes $\{\Lambda_\varepsilon\}_{\varepsilon>0}$.
Using well-established properties of dimension theory in uniform hyperbolic systems, one can show that
\[ \mathrm{Dim} (\Lambda_\varepsilon\cap W^i_\beta(f,x)) \approx \frac{h_\mu(f)}{|\lambda_i(\mu)|}\]
for $i=u, s$ and every $x\in\Lambda$.
Burns and Wilkinson \cite{bw2005} proved the holonomy maps of the stable and unstable foliations for $(f, \Lambda_\varepsilon)$ are Lipschitz continuous.
Consequently, one can show that
\begin{eqnarray*}
\begin{aligned}
     &\dim_H(\Lambda_\varepsilon \cap W^u_\beta(f, x)) + \dim_H(\Lambda_\varepsilon \cap W^s_\beta(f, x))\\
\leq &\mathrm{Dim}\Lambda_\varepsilon\\
\leq &\overline{\dim}_B(\Lambda_\varepsilon \cap W^u_\beta(f, x)) + \overline{\dim}_B(\Lambda_\varepsilon \cap W^s_\beta(f, x))
\end{aligned}
\end{eqnarray*} for every $x\in\Lambda_\varepsilon$. Hence, $\mathrm{Dim}\mu$ is approximately equal to $\mathrm{Dim}\Lambda_\varepsilon$. The detailed proofs will be given in the next section.

\begin{maintheorem}\label{B}
Let $f:M\to M$ be a $C^{1+\alpha}$ diffeomorphism, and $\mu$ be an $f$-invariant ergodic hyperbolic measure with $h_\mu(f)>0$. Assume that $(f,\mu)$ has only two Lyapunov exponents $\lambda_u(\mu)>0>\lambda_s(\mu)$. For each $\varepsilon>0$, there exists a horseshoe $\Lambda_\varepsilon$ such that
\[
   |\mathrm{ Dim} \Lambda_\varepsilon- \mathrm{ Dim} \mu|<\varepsilon
\]
where $\mathrm{ Dim}$ denotes either $\dim_H$ or $\underline{\dim}_B$ or $\overline{\dim}_B$.
\end{maintheorem}

 In \cite{wcz},  the authors relaxed the smoothness of Theorem \ref{horse-app} to $C^1$ under the additional condition that Oseledec's splitting $E^u\oplus E^s$ of $(f,\mu)$ is dominated. In this setting, one do not have Lipschitz continuity of the holonomy map in general. However, using Palis and Viana's method \cite{PV88} one can show that: for every $\gamma\in(0, 1)$, there is some $D_\gamma>0$ such that the holonomy maps of the stable and unstable foliations for the hyperbolic dynamical system $(f, \Lambda_\varepsilon)$ (See Lemma \ref{holderfoliation}) are $(D_\gamma, \gamma)$-H\"{o}lder continuous. Since $\gamma$ is arbitrary, using the ideas in  \cite{wwcz} one can prove the following theorem:

\begin{maintheorem}\label{C}
Let $f:M\to M$ be a $C^{1}$ diffeomorphism, and let $\mu$ be an $f$-invariant ergodic hyperbolic measure with $h_\mu(f)>0$. Assume that $(f,\mu)$ has only two Lyapunov exponents $\lambda_u(\mu)>0>\lambda_s(\mu)$ and the corresponding Oseledec's splitting $E^u\oplus E^s$ is dominated. For each $\varepsilon>0$, there exists a horseshoe $\Lambda_\varepsilon$ such that
\[
   |\mathrm{ Dim} \Lambda_\varepsilon- \mathrm{ Dim} \mu|<\varepsilon
\]
where $\mathrm{ Dim}$ denotes either $\dim_H$ or $\underline{\dim}_B$ or $\overline{\dim}_B$.
\end{maintheorem}


 \section{Proofs}\label{proof}
 In this section, we provide the proof of the main results presented in the previous section.

\subsection{Proof of Theorem \ref{A}}Given an $f$-invariant ergodic measure $\mu$, let $P(t):=P_\mu(f|_\Lambda, \Phi_f(t))$, it is easy to see that the function $t\mapsto P(t)$ is continuous and strictly decreasing on the interval $[0,m_0]$. It follows from \eqref{metric-ent-form} that $P(0)=h_\mu(f)\ge 0$, and $P(m_0)\le 0$ by Margulis-Ruelle's inequality. Consequently,  there exists a unique root $s_\mu$  of the equation $P_\mu(f|_\Lambda, \Phi_f(t))=0$.

 If $h_\mu(f)=0$, it is easy to see that $h_\mu(f)=s_\mu=0$. Hence, $\dim_L\mu=s_\mu$.

 If $0<h_\mu(f)< \lambda_{m_0}(\mu)$, then $P(0)>0$ and $P(1)<0$. This implies that $s_\mu\in (0,1)$ and $0=P(s_\mu)=h_\mu(f)-s_\mu \lambda_{m_0}(\mu)$.  In consequence, we have that $$s_\mu=\dim_L\mu=\frac{h_\mu(f)}{\lambda_{m_0}(\mu)}.$$

If $h_\mu(f)\ge \lambda_{m_0}(\mu)$, note that
$$
\begin{aligned}
0&=h_\mu(f)+\mathcal{L}_*(\Phi_f(s_\mu),\mu)\\
&=h_\mu(f)-\sum_{i=m_0-[s_\mu]+1}^{m_0}\lambda_i(\mu)-(s_\mu-[s_\mu])\lambda_{m_0-[s_\mu]}(\mu).
\end{aligned}
$$
Hence,
$$
s_\mu=[s_\mu]+\frac{h_\mu(f)-\sum_{i=m_0-[s_\mu]+1}^{m_0}\lambda_i(\mu)}{\lambda_{m_0-[s_\mu]}(\mu)}.
$$
On the other hand, since $t\mapsto P(t)$ is  strictly decreasing in $t$,  we have that
\[
[s_\mu]=\max\{i:  \lambda_{m_0}(\mu)+\cdots+\lambda_{m_0-i +1}(\mu)\le h_{\mu}(f)\}.
\]
This yields that
\[
s_\mu=\dim_{\mathrm{L}}\mu.
\]


To prove the second statement, by Theorem 5.1 in \cite{cpz}, for each $f$-invariant ergodic measure $\mu$ with positive entropy, and for each $\varepsilon>0$ there exists an $f$-invariant compact subset $\Lambda_\varepsilon\subset\Lambda$ such that the following statements hold:
\begin{enumerate}
\item[(i)] $h_{\text{top}}(f|_{\Lambda_\varepsilon})\ge h_{\mu}(f)-\varepsilon$;
\item[(ii)] there is a continuous invariant splitting
$T_x M=E_1(x)\oplus E_2(x)\oplus\cdots\oplus E_\ell(x)$ over
$\Lambda_\varepsilon$ and a constant $C>0$ so that
\[
C^{-1}\exp(n(\lambda_j(\mu)-\varepsilon))\le \|D_xf^n(u)\| \le C\exp(n(\lambda_j(\mu)+\varepsilon))
\]for any unit vector $u\in E_j(x)$,
where $\lambda_1(\mu)<\dots<\lambda_\ell(\mu)$ are distinct Lyapunov exponents of $f$ with respect to the measure $\mu$.
\end{enumerate}
By  modifying the arguments in \cite[Theorem 5.1]{cpz},  one may improve the estimate in (i)  as follows:
\begin{enumerate}
\item[(i)$^{\prime}$] $h_{\mu}(f)+\varepsilon\ge h_{\text{top}}(f|_{\Lambda_\varepsilon})\ge h_{\mu}(f)-\varepsilon$.
\end{enumerate}
Since $\Lambda_{\varepsilon}$ is a repeller of $f$, one can choose an $f$-invariant ergodic measure $\mu_\varepsilon$ on $\Lambda_\varepsilon$ so that $h_{\mu_\varepsilon}(f)=h_{\text{top}}(f|_{\Lambda_\varepsilon})$, it yields that
\[
\begin{aligned}
P_{\mathrm{top}}(f|_{\Lambda_\varepsilon}, \Phi_f(t))&\ge h_{\mu_\varepsilon}(f)+\mathcal{L}_*(\Phi_f(t), \mu_\varepsilon)\\
&\ge h_{\mu}(f)+\mathcal{L}_*(\Phi_f(t), \mu)-(t+1)\varepsilon\qquad (~\text{by (ii)}~)\\
&\ge P_{\mu}(f|_{\Lambda}, \Phi_f(t))-(m_0+1)\varepsilon.
\end{aligned}
\]
On the other hand, since $f$ is expanding, by the variational principle there exists an $f$-invariant ergodic measure $\widetilde{\mu}_\varepsilon$ on $\Lambda_\varepsilon$ so that
\[
\begin{aligned}
P_{\mathrm{top}}(f|_{\Lambda_\varepsilon}, \Phi_f(t))&=h_{\widetilde{\mu}_\varepsilon}(f)+\mathcal{L}_*(\Phi_f(t), \widetilde{\mu}_\varepsilon)\\
&\le h_\mu(f)+\mathcal{L}_*(\Phi_f(t), \mu)+(t+1)\varepsilon\\
&\le  P_{\mu}(f|_{\Lambda}, \Phi_f(t))+(m_0+1)\varepsilon.
\end{aligned}
\]
Hence,
\[
\Big|P_{\mathrm{top}}(f|_{\Lambda_\varepsilon}, \Phi_f(t))-P_{\mu}(f|_{\Lambda}, \Phi_f(t))\Big|\le(m_0+1)\varepsilon.
\]
By Theorem 4.1 in \cite{cpz}, the Carath\'{e}odory singular dimension $\dim_C \Lambda_\varepsilon$ of $\Lambda_\varepsilon$ is given by the unique root of the following equation
\[
P_{\mathrm{top}}(f|_{\Lambda_\varepsilon}, \Phi_f(t))=0.
\]
This together with the first statement yield that
\[
\begin{aligned}
K|\dim_C \Lambda_\varepsilon-\dim_\mathrm{L}\mu|&\le \Big|P_{\mu}(f|_{\Lambda}, \Phi_f(\dim_C \Lambda_\varepsilon))-P_{\mu}(f|_{\Lambda}, \Phi_f(\dim_\mathrm{L}\mu))\Big |\\
&=\Big|P_{\mu}(f|_{\Lambda}, \Phi_f(\dim_C \Lambda_\varepsilon))-P_{\mathrm{top}}(f|_{\Lambda_\varepsilon}, \Phi_f(\dim_C \Lambda_\varepsilon))\Big |\\
&\le(m_0+1)\varepsilon
\end{aligned}
\]
where $K=\min_{x\in \Lambda}m(D_xf)$ and $m(\cdot)$ denotes the minimum norm of an operator. Consequently, we have that $\dim_C \Lambda_\varepsilon\rightarrow \dim_\mathrm{L}\mu$ as $\varepsilon$ approaching zero.

\subsection{Proof of Proposition \ref{dimequal}}
 Given a subset $Z\subset \Lambda$, since $P_Z(f|_{\Lambda},\Phi_f(t))$ is continuous and strictly decreasing in $t$, let $t_Z$ denote the unique root of the equation $P_Z(f|_{\Lambda},\Phi_f(t))=0$. For every $t<t_Z$, we have that $P_Z(f|_{\Lambda},\Phi_f(t))>0$. Fix such a number $t$, and take $\beta>0$ so that $P_Z(f|_{\Lambda},\Phi_f(t))-\beta>0$. Since $$\displaystyle{P_Z(f|_{\Lambda},\Phi_f(t))=\liminf_{r\to 0} P_Z(f|_{\Lambda},\Phi_f(t),r)}$$there exists $r_0>0$ such that for each $0<r<r_0$ one has
\[
P_Z(f|_{\Lambda},\Phi_f(t),r)>P_Z(f|_{\Lambda},\Phi_f(t))-\beta.
\]
Fix such a small $r>0$. By the definition of topological pressure on  arbitrary subsets, one has
\[
m(Z,\Phi_f(t),P_Z(f|_{\Lambda},\Phi_f(t))-\beta,r)=+\infty.
\]
Hence, for each $\xi>0$, there exists $L\in \mathbb{N}$ so that for any $N>L$ we have that
\[
\begin{aligned}
&\exp(-N(P_Z(f|_{\Lambda},\Phi_f(t))-\beta))\inf \Big\{ \sum_{i}\exp (\sup_{y\in B_{n_i}(x_i,r)}-\varphi^t(y, f^{n_i}))  \Big\}
\\
&\ge\inf\Big\{ \sum_{i}\exp (-(P_Z(f|_{\Lambda},\Phi_f(t))-\beta)n_i +\sup_{y\in B_{n_i}(x_i,r)}-\varphi^t(y, f^{n_i}))  \Big\}>\xi
\end{aligned}
\]
where the infimum is taken over all collections $\{B_{n_i}(x_i,r)\}$ of Bowen's balls with $n_i\ge N$, which covers $Z$. This yields that
\[
\inf \Big\{ \sum_{i}\exp (\sup_{y\in B_{n_i}(x_i,r)}-\varphi^t(y, f^{n_i}))  \Big\}>\xi \exp(N(P_Z(f|_{\Lambda},\Phi_f(t))-\beta)).
\]
Letting $N\to\infty$, we have that
\[
m(Z,t, r)=+\infty.
\]
Hence,
\begin{eqnarray*}
\dim_{C,r}Z\ge t
\end{eqnarray*}
for all $0<r<r_0$. Consequently, since $t<t_Z$ is arbitrary, we have that
\begin{eqnarray}\label{LB}
\dim_{C}Z\ge t_Z.
\end{eqnarray}

On the other hand, for each $t>t_Z$ one has that $P_Z(f|_{\Lambda},\Phi_f(t))<0$. Fix such a number $t$, and take $\widetilde{\beta}>0$ so that $P_Z(f|_{\Lambda},\Phi_f(t))+\widetilde{\beta}<0$.  By the definition of topological pressure on arbitrary subsets, for any $R>0$, there exists $0<r<R$ such that
\[
P_Z(f|_{\Lambda},\Phi_f(t),r)<P_Z(f|_{\Lambda},\Phi_f(t))+\widetilde{\beta}.
\]
For such a small $r>0$ one has
\[
m(Z,\Phi_f(t),P_Z(f|_{\Lambda},\Phi_f(t))+\widetilde{\beta},r)=0
\]
Hence, for each small $\widetilde{\xi}>0$ there exists $\widetilde{L}\in \mathbb{N}$ so that for any $N>\widetilde{L}$ we have that
\[
\begin{aligned}
&\exp(-N(P_Z(f|_{\Lambda},\Phi_f(t))+\widetilde{\beta}))\inf \Big\{ \sum_{i}\exp (\sup_{y\in B_{n_i}(x_i,r)}-\varphi^t(y, f^{n_i}))  \Big\}
\\
&\le\inf\Big\{ \sum_{i}\exp (-(P_Z(f|_{\Lambda},\Phi_f(t))+\widetilde{\beta})n_i +\sup_{y\in B_{n_i}(x_i,r)}-\varphi^t(y, f^{n_i}))  \Big\}\le \widetilde{\xi}
\end{aligned}
\]
where the infimum is taken over all collections $\{B_{n_i}(x_i,r)\}$ of Bowen's balls with $n_i\ge N$, which covers $Z$. This yields that
\[
\inf \Big\{ \sum_{i}\exp (\sup_{y\in B_{n_i}(x_i,r)}-\varphi^t(y, f^{n_i}))  \Big\}\le \widetilde{\xi} \exp(N(P_Z(f|_{\Lambda},\Phi_f(t))+\widetilde{\beta}))
\]
 Letting $N\to\infty$, one has
\[
m(Z, t,r)=0.
\]
Consequently, for such $r>0$ one has
\[
\dim_{C, r} Z\le t.
\]
Hence, we have that
\begin{eqnarray}\label{UB}
\dim_{C} Z=\liminf_{r\to 0}\dim_{C, r} Z\le t_Z.
\end{eqnarray}
It follows from \eqref{LB} and \eqref{UB} that $$\dim_CZ=t_Z.$$

To show the second statement, for a given $f$-invariant ergodic measure $\mu$ supported on $\Lambda$, and a subset $Z\subset \Lambda$ with $\mu(Z)=1$, we have that
\[
P_Z(f|_{\Lambda}, \Phi_f(t))\ge P_\mu(f|_{\Lambda}, \Phi_f(t)).
\]
By (1) of Theorem \ref{A} and the first statement, one has
\[
\dim_C Z\ge \dim_L\mu.
\]
By the definition of Carath\'{e}odory singular dimension of arbitrary subsets, one has
\[
\dim_{C,r} Z\ge \dim_L\mu
\]
for all sufficiently small $r>0$. Consequently, we have that
\[
\dim_C\mu=\liminf_{r\to 0}\dim_{C,r} \mu =\liminf_{r\to 0}\inf\{\dim_{C,r} Z: \mu(Z)=1 \}\ge \dim_L\mu.
\]
To prove that $\dim_C\mu=\dim_L\mu$, we assume that $\dim_C\mu>\widetilde{t}>\dim_L\mu$. By the first statement in Theorem \ref{A}, we have that
\[
P_\mu(f|_{\Lambda},\Phi(\widetilde{t}))<0.
\]
By the definition of measure theoretic pressure, for each $n\in \mathbb{N}$, there exists $0<r_n<\frac 1n$ so that
\[
\inf\{P_Z(f|_{\Lambda},\Phi(\widetilde{t}),r_n): \mu(Z)=1\}<0.
\]
Hence, there exists a subset $Z_n\subset \Lambda$ with $\mu(Z_n)=1$ so that
\[
P_{Z_n}(f|_{\Lambda},\Phi(\widetilde{t}),r_n)<0.
\]
Put $\widetilde{Z}:=\bigcap_{n\ge 1} Z_n$, then $\mu(\widetilde{Z})=1$  and
\[
\begin{aligned}
P_{\widetilde{Z}}(f|_{\Lambda},\Phi(\widetilde{t}))&=\liminf_{r\to 0}P_{\widetilde{Z}}(f|_{\Lambda},\Phi(\widetilde{t}),r)\\
&\le \liminf_{n\to \infty}P_{Z_n}(f|_{\Lambda},\Phi(\widetilde{t}),r_n)\le 0
\end{aligned}
\]
It follows from the first statement and the definition of Carath\'{e}odory singular dimension of $\mu$ that
\[
\dim_C\mu=\liminf_{r\to 0} \dim_{C,r}\mu \le \liminf_{r\to 0} \dim_{C,r} \widetilde{Z}=\dim_C\widetilde{Z}\le \widetilde{t},
\]
which yields a contraction.  Hence, we have that $\dim_C\mu=\dim_L\mu$.

\subsection{Proof of Theorem \ref{B}}
Ledrappier, Young \cite{ly1} and Barreira, Pesin, Schmeling \cite{bps} proved that
\[ \mathrm{Dim}\mu = \frac{h_\mu(f)}{\lambda_u(\mu)} - \frac{h_\mu(f)}{\lambda_s(\mu)}\]
where $\mathrm{ Dim}$ denotes either $\dim_H$ or $\underline{\dim}_B$ or $\overline{\dim}_B$.
Fix a small number $\varepsilon>0$. By Theorem \ref{horse-app}, there exists a horseshoe $\Lambda_\varepsilon$ such that
\begin{enumerate}
\item[(i)] $|h_{\mathrm{top}}(f,\Lambda_\varepsilon) - h_{\mu}(f)|<\varepsilon$;
\item[(ii)] there exists a dominated splitting $T_{\Lambda_\varepsilon} M=E^u\oplus E^s$ with $\dim E^i=d_i\, (i=u,s)$, and for each $x\in \Lambda_\varepsilon$, every $n\geq 1$ and each vector $v\in E^i(x)$ $(i=s,u)$
     \[
     e^{(\lambda_i(\mu)-\varepsilon)n}< \|D_xf^{n}(v)\|< e^{(\lambda_i(\mu)+\varepsilon)n}.
     \]
\end{enumerate}
 Fixed any $k\in\mathbb{N}$, denote $F=f^{2^k}$. Since $\Lambda_\varepsilon$ is a locally maximal hyperbolic set for $f$, $\Lambda_\varepsilon$ is also a locally maximal hyperbolic set for $F$. Notice that
\[ W^u_\beta(F, x)\cap\Lambda_\varepsilon = W^u_\beta(f, x)\cap\Lambda_\varepsilon \quad \text{and} \quad W^s_\beta(F, x)\cap\Lambda_\varepsilon = W^s_\beta(f, x)\cap\Lambda_\varepsilon.\]
Let $\|\cdot\|$ and $m(\cdot)$ denote the maximal and minimal norm of an operator.  For every $x\in\Lambda_\varepsilon$, Barreira \cite{ba96} proved that
\begin{eqnarray*}
\underline{t}_u^k \leq \dim_H(\Lambda_\varepsilon\cap W^u_\beta(f,x))
                  \leq \overline{\dim}_B(\Lambda_\varepsilon\cap W^u_\beta(f,x))
                  \leq \overline{t}_u^k
\end{eqnarray*}
where $\underline{t}_u^k$, $\overline{t}_u^k$ are the unique solutions of
\[P_{\mathrm{top}}(F, -t\log\|D_xF|_{E^u(x)}\|)=0 \quad\text{and}\quad  P_{\mathrm{top}}(F, -t\log m(D_xF|_{E^u(x)}))=0\]
respectively. Using the same arguments as in the proof of Theorem $6.2$ and Theorem $6.3$ in \cite{bch},
one can prove that the sequences $\{\underline{t}_u^k\}$ and $\{\overline{t}_u^k\}$ are monotone. Furthermore, set
\[\underline{t}_u:=\lim_{k\to\infty}\underline{t}_u^k   \quad\text{and}\quad  \overline{t}_u:= \lim_{k\to\infty}\overline{t}_u^k,\]
one can show that $\underline{t}_u$, $\overline{t}_u$ are the unique solutions of the following equations
\[P_{\mathrm{var}}(f, -t\{\log\|D_xf^n|_{E^u}\|\})=0, \quad P_{\mathrm{top}}(f, -t\{\log m(D_xf^n|_{E^u})\})=0\]
respectively.

Consequently, we have that
\begin{equation*}
\underline{t}_u \leq \dim_H(\Lambda_\varepsilon\cap W^u_\beta(f,x)) \leq \underline{\dim}_B(\Lambda_\varepsilon\cap W^u_\beta(f,x)) \leq \overline{\dim}_B(\Lambda_\varepsilon\cap W^u_\beta(f,x)) \leq \overline{t}_u
\end{equation*}
and
\begin{eqnarray*}
\begin{aligned}
\underline{t}_u &=\sup\Big\{\frac{h_\nu(f)}{\displaystyle{\lim_{n\to\infty}} \frac 1n \int \log\|D_xf^n|_{E^u}\|d\nu}: \nu\in\mathcal{M}_f(\Lambda_\varepsilon) \Big\},\\
\overline{t}_u  &=\sup\Big\{\frac{h_\nu(f)}{\displaystyle{\lim_{n\to\infty}} \frac 1n \int \log m(D_xf^n|_{E^u})d\nu}: \nu\in\mathcal M_f(\Lambda_\varepsilon) \Big\}.
\end{aligned}
\end{eqnarray*}
Combining with (i) and (ii), one has
\begin{eqnarray}\label{0star}
\frac{h_\mu(f)-\varepsilon}{\lambda_u(\mu)+\varepsilon} \le \mathrm{Dim}(\Lambda_\varepsilon \cap W^u_\beta(f,x))\le \frac{h_\mu(f)+\varepsilon}{\lambda_u(\mu)-\varepsilon}
\end{eqnarray}
for every $x\in \Lambda_\varepsilon$, where $\mathrm{Dim}$ denotes either $\dim_H$ or $\underline{\dim}_B$ or $\overline{\dim}_B$.  One can show in a similar fashion that
\begin{eqnarray}\label{1star}
\begin{aligned}
-\frac{h_\mu(f)-\varepsilon}{\lambda_s(\mu)-\varepsilon}\le \mathrm{Dim}(\Lambda_\varepsilon \cap W^s_\beta(f,x)) \le  -\frac{h_\mu(f)+\varepsilon}{\lambda_s(\mu)+\varepsilon}
\end{aligned}
\end{eqnarray}
for every $x\in \Lambda_\varepsilon$.

\begin{lemma}\label{holonomy map of w}
The holonomy maps of the stable and unstable foliations for $(f, \Lambda_\varepsilon)$ are Lipschitz continuous.
\end{lemma}
\begin{proof}
Fix a positive integer $N$, put $F:=f^N$ and $\Lambda=\Lambda_\varepsilon$. Since $\Lambda$ is a locally maximal hyperbolic set for $f$, so is $\Lambda$ for $F$. Notice that
\begin{equation}\label{foliationFf}
W^u_\beta(F, x)\cap\Lambda = W^u_\beta(f, x)\cap\Lambda \quad \text{and} \quad W^s_\beta(F, x)\cap\Lambda = W^s_\beta(f, x)\cap\Lambda.
\end{equation}
Let
$$a_F=\|DF^{-1}|_{E^u}\|,\ b_F=\|DF|_{E^s}\|,c_F=\|DF|_{E^u}\|, \ d_F=\|DF^{-1}|_{E^s}\|.$$
It follows from (ii) that
\[
1<\frac{\|D_xF|_{E^i(x)}\|}{m(D_xF|_{E^i(x)})}<e^{2\varepsilon N}, \text{ for every } x\in \Lambda~~\text{and}~~i\in\{s,u\}.
\]
Hence,
\[
a_F b_F c_F=\frac{\|DF|_{E^s}\|\cdot \|DF|_{E^u}\|}{m(DF|_{E^u})}<e^{(\lambda_s(\mu)+3\varepsilon)N}<1
\]
provided that $\varepsilon>0$ is sufficiently small such that $\lambda_s(\mu)+3\varepsilon<0$. By Theorem $0.2$ in \cite{bw2005}, we have that the holonomy map of the stable foliation for $(F, \Lambda)$ is $C^1$. Similarly, note that
\[
a_Fb_Fd_F=\frac{\|DF|_{E^s}\|}{m(DF||_{E^s})m(DF||_{E^u})}<e^{(-\lambda_u(\mu)+3\varepsilon)N}<1
\]
provided that $\varepsilon>0$ is sufficiently small such that $\lambda_u(\mu)-3\varepsilon>0$. It follows from \cite[Theorem 0.2]{bw2005} that the holonomy map of the unstable foliation for $(F, \Lambda)$ is $C^1$. Combing (\ref{foliationFf}) one has the holonomy maps of the stable and unstable foliations for $(f, \Lambda)$ are Lipschitz continuous.
\end{proof}

By Lemma \ref{holonomy map of w} and the fact $f$ is topologically mixing on $\Lambda_\varepsilon$, one has $\dim_H(\Lambda_\varepsilon\cap W^u_\beta(f, x))$, $\underline{\dim}_B(\Lambda_\varepsilon\cap W^u_\beta(f, x))$ and $\overline{\dim}_B(\Lambda_\varepsilon\cap W^u_\beta(f,x))$ are independent of $\beta$ and $x$ (see the proof of Theorem $4.3.2$ in \cite{ba2008} for more details).
Let
\[A_{\varepsilon, x}=(\Lambda_\varepsilon \cap W^u_\beta(f,x)) \times (\Lambda_\varepsilon \cap W^s_\beta(f,x)).\]
 By the properties of dimension (e.g. see \cite{fal03,Pes97}), one has
\begin{eqnarray}\label{2star}
\begin{aligned}
     & \dim_H(\Lambda_\varepsilon\cap W^u_\beta(f, x)) + \dim_H(\Lambda_\varepsilon\cap W^s_\beta(f, x))\\
\leq\ & \dim_H A_{\varepsilon, x}\\
\leq\ & \underline{\dim}_B A_{\varepsilon, x}\\
\leq\ & \overline{\dim}_B A_{\varepsilon, x}\\
\leq\ & \overline{\dim}_B (\Lambda_\varepsilon\cap W^u_\beta(f, x)) + \overline{\dim}_B (\Lambda_\varepsilon\cap W^s_\beta(f, x)).
\end{aligned}
\end{eqnarray}
Let $\Phi: A_{\varepsilon,x} \to \Lambda_\varepsilon$ be given by
\[ \Phi(y,z)=W^s_\beta(f, y) \cap W^u_\beta(f, z).\]
It is easy to see $\Phi$ is a homeomorphism onto a neighborhood $V_x$ of $x$ in $\Lambda_\varepsilon$. It follows from Lemma \ref{holonomy map of w} that $\Phi$ and $\Phi^{-1}$ are Lipschitz continuous (see  Theorem $4.3.2$ in \cite{ba2008} for detailed proofs). It follows from Corollary \ref{lipdimension} that
\[ \mathrm{ Dim} V_x = \mathrm{ Dim} A_{\varepsilon, x}, \]
where $\mathrm{ Dim}$ denotes either $\dim_H$ or $\underline{\dim}_B$ or $\overline{\dim}_B$.
Since $\{V_x: x\in \Lambda_\varepsilon\}$ is an open cover of $\Lambda_\varepsilon$, one can choose a finite open cover $\{V_{x_1}, V_{x_2}, \cdots, V_{x_k}\}$ of $\Lambda_\varepsilon$. It follows from (\ref{2star}) that
\begin{eqnarray*}
\begin{aligned}
     & \dim_H(\Lambda_\varepsilon \cap W^u_\beta(f, x)) + \dim_H(\Lambda_\varepsilon \cap W^s_\beta(f, x))\\
\leq\ & \dim_H \Lambda_\varepsilon = \max_{1\leq i \leq k} \dim_H V_{x_i}\\
\leq\ & \overline{\dim}_B \Lambda_\varepsilon =\max_{1\leq i\leq k}\overline{\dim}_B V_{x_i}\\
\leq\ & \overline{\dim}_B(\Lambda_\varepsilon \cap W^u_\beta(f, x)) + \overline{\dim}_B(\Lambda_\varepsilon \cap W^s_\beta(f, x)),
\end{aligned}
\end{eqnarray*}
for every $x\in\Lambda_\varepsilon$.
Combining \eqref{0star} and \eqref{1star}  we obtain
\[\lim_{\varepsilon\to0} \mathrm{ Dim} \Lambda_\varepsilon = \frac{h_\mu(f)}{\lambda_u(\mu)} - \frac{h_\mu(f)}{\lambda_s(\mu)} = \mathrm{ Dim} \mu, \]
where $\mathrm{ Dim}$ denotes either $\dim_H$ or $\underline{\dim}_B$ or $\overline{\dim}_B$.
This completes the proof of Theorem \ref{B}.


\subsection{Proof of Theorem \ref{C}}
For every pair $(f,\mu)$ satisfying the assumptions,
Wang and Cao \cite[Corollary 1]{wc2016} proved that
\[ \dim_H\mu = \frac{h_\mu(f)}{\lambda_u(\mu)} - \frac{h_\mu(f)}{\lambda_s(\mu)}.\]
Fix a small number $\varepsilon>0$. Wang, Cao and Zou \cite[Theorem 1.1]{wcz} proved that there exists a horseshoe $\Lambda_\varepsilon$ such that
\begin{enumerate}
\item[(i)]$|h_{\mathrm{top}}(f,\Lambda_\varepsilon) - h_{\mu}(f)|<\varepsilon$;
\item[(ii)] there exists a dominated splitting $T_{\Lambda_\varepsilon} M=E^u\oplus E^s$ with $\dim E^i=d_i\, (i=u,s)$, and for each $x\in \Lambda_\varepsilon$, every $n\geq 1$ and each vector $v\in E^i(x)$ $(i=s, u)$,
     \[
     e^{(\lambda_i(\mu)-\varepsilon)n}< \|D_xf^{n}(v)\|< e^{(\lambda_i(\mu)+\varepsilon)n}.
     \]
\end{enumerate}
Fix a positive integer $k\in\mathbb{N}$, denote $F=f^{2^k}$. Since $\Lambda_\varepsilon$ is a locally maximal hyperbolic set for $f$, so is $\Lambda_\varepsilon$ for $F$. Notice that
\[ W^u_\beta(F, x)\cap\Lambda_\varepsilon = W^u_\beta(f, x)\cap\Lambda_\varepsilon \quad \text{and} \quad W^s_\beta(F, x)\cap\Lambda_\varepsilon = W^s_\beta(f, x)\cap\Lambda_\varepsilon.\]
For every $x\in\Lambda_\varepsilon$, it follows from \cite[Lemmas 3.5 and 3.6]{wwcz} that
\begin{eqnarray*}
\underline{t}_u^k \leq \dim_H(\Lambda_\varepsilon\cap W^u_\beta(f,x))
                  \leq \overline{\dim}_B(\Lambda_\varepsilon\cap W^u_\beta(f,x))
                  \leq \overline{t}_u^k,
\end{eqnarray*}
where $\underline{t}_u^k$, $\overline{t}_u^k$ are the unique roots of
\[P_{\mathrm{top}}(F, -t\log\|D_xF|_{E^u(x)}\|)=0, \quad P_{\mathrm{top}}(F, -t\log m(D_xF|_{E^u(x)}))=0\]
respectively. Using the same arguments as in the proof of Theorem $6.2$ and Theorem $6.3$ in \cite{bch},
one can prove that the sequences $\{\underline{t}_u^k\}$ and $\{\overline{t}_u^k\}$ are monotone. Set
\[\underline{t}_u:=\lim_{k\to\infty}\underline{t}_u^k \quad\text{and}\quad  \overline{t}_u:= \lim_{k\to\infty}\overline{t}_u^k,\]
where $\underline{t}_u$, $\overline{t}_u$ are the unique solutions of the following equations
\[P_{\mathrm{var}}(f, -t\{\log\|D_xf^n|_{E^u(x)}\|\})=0, \quad P_{\mathrm{top}}(f, -t\{\log m(D_xf^n|_{E^u(x)})\})=0\]
respectively. Hence, we have that
\begin{equation*}
\underline{t}_u \leq \dim_H(\Lambda_\varepsilon\cap W^u_\beta(f,x)) \leq \underline{\dim}_B(\Lambda_\varepsilon\cap W^u_\beta(f,x)) \leq \overline{\dim}_B(\Lambda_\varepsilon\cap W^u_\beta(f,x)) \leq \overline{t}_u.
\end{equation*}
Since
\begin{eqnarray*}
\underline{t}_u &=\sup\Big\{\frac{h_\nu(f)}{ \displaystyle{\lim_{n\to\infty}}\frac 1n \int \log\|D_xf^n|_{E^u(x)}\|d\nu}: \nu\in\mathcal{M}_f(\Lambda_\varepsilon) \Big\}
\end{eqnarray*}
and
\begin{eqnarray*}
\overline{t}_u  &=\sup\Big\{\frac{h_\nu(f)}{\displaystyle{\lim_{n\to\infty}} \frac 1n \int \log m(D_xf^n|_{E^u(x)})d\nu}: \nu\in\mathcal{M}_f(\Lambda_\varepsilon) \Big\}
\end{eqnarray*}
using (i) and (ii) one can show that
\begin{eqnarray}\label{4star}
\frac{h_\mu(f)-\varepsilon}{\lambda_u(\mu)+\varepsilon}\le \mathrm{Dim}(\Lambda_\varepsilon \cap W^u_\beta(f,x))\le \frac{h_\mu(f)+\varepsilon}{\lambda_u(\mu)-\varepsilon}
\end{eqnarray}
for every $x\in \Lambda_\varepsilon$, where $\mathrm{Dim}$ denotes either $\dim_H$ or $\underline{\dim}_B$ or $\overline{\dim}_B$. Similarly, we obtain that
\begin{eqnarray}\label{5star}
-\frac{h_\mu(f)-\varepsilon}{\lambda_s(\mu)-\varepsilon}\le \mathrm{Dim}(\Lambda_\varepsilon \cap W^s_\beta(f,x)) \le  -\frac{h_\mu(f)+\varepsilon}{\lambda_s(\mu)+\varepsilon}
\end{eqnarray}
for every $x\in \Lambda_\varepsilon$.

\begin{lemma}\label{holderfoliation}
Let $\Lambda$ be a locally maximal hyperbolic set of a $C^1$ diffeomorphism such that $f$ is topologically mixing on $\Lambda$. Assume that the diffeomorphism $f|_\Lambda$ possesses a $\{\lambda_u(\mu), \lambda_s(\mu)\}$-dominated splitting $T_\Lambda M=E^u\oplus E^s$ with $E^u\succeq E^s$ and $\lambda_u(\mu)>0>\lambda_s(\mu)$. Then for every $\gamma\in(0, 1)$ there exists $D_\gamma>0$ such that the holonomy maps of the stable and unstable foliations for $f$ are $(D_\gamma, \gamma)$-H$\ddot{o}$lder continuous.
\end{lemma}

\begin{proof}
For every $x\in\Lambda$, $n\in\mathbb{N}$ and each unit vector $v\in E^i(x)$ ($i=u, s$),
\begin{equation*}
e^{n(\lambda_i(\mu)-\varepsilon)} \leq \|D_xf^{n}(v)\| \leq  e^{n(\lambda_i(\mu)+\varepsilon)}.
\end{equation*}
Fix a positive integer $N$, put $F:=f^N$.
This implies that for every $x\in\Lambda$,
\begin{equation*}
1 \leq \frac{\|D_xF|_{E^u}\|}{m(D_xF|_{E^u})} \leq e^{2N\varepsilon},\quad 1 \leq \frac{\|D_xF|_{E^s}\|}{m(D_xF|_{E^s})} \leq e^{2N\varepsilon}.
\end{equation*}
Notice that $\Lambda$ is also a locally maximal hyperbolic set for $F$ and
\[ W^u_\beta(F, x)\cap\Lambda = W^u_\beta(f, x)\cap\Lambda \quad \text{and} \quad W^s_\beta(F, x)\cap\Lambda = W^s_\beta(f, x)\cap\Lambda\]
for every $x\in\Lambda$.

Let $\pi^s$ and $\pi^u$ be the holonomy maps of stable and unstable foliations for $f$, i.e. for any $x\in\Lambda$, $x'\in W^s_\beta(f, x)$ and $x''\in W^u_\beta(f, x)$ close to $x$,
$$\pi^s:\ W^u_\beta(f,x)\cap\Lambda\to W^u_\beta(f,x')\cap\Lambda\ \text{with}\ \pi^s(y)=W^s_\beta(f,y)\cap W^u_\beta(f,x')$$
and
$$\pi^u:\ W^s_\beta(f,x)\cap\Lambda\to W^s_\beta(f,x'')\cap\Lambda\ \text{with}\ \pi^u(z)=W^u_\beta(f,z)\cap W^s_\beta(f,x'').$$
Therefore, $\pi^s$ is also a map from $W^u_\beta(F,x)\cap\Lambda$ to $W^u_\beta(F, x')\cap\Lambda$ and $\pi^u$ is also a map from $W^s_\beta(F,x)\cap\Lambda$ to $W^s_\beta(F, x'')\cap\Lambda$.

Let $U\subset M$ be an open subset such that $\Lambda=\bigcap_{n\in\mathbb{Z}}f^n(U)$,  and $\mathcal{U}\subset \mbox{Diff}^{1}(M)$ be a neighbourhood of $f$ such that, for each $g \in \mathcal{U}$, $\Lambda_g = \bigcap_{n\in\mathbb{Z}} g^n(U)$ is a locally maximal hyperbolic set for $g$ and there is a homeomorphism $h_g: \Lambda \rightarrow \Lambda_g$ satisfies that $g\circ h_g=h_g\circ f$, with $h_g$ $C^0-$close to identity if $g$ is $C^1-$close to $f$.
For $g \in \mathcal{U}$, let $T_{\Lambda_g}M=E^u_g\oplus E^s_g$ denote the hyperbolic splitting over $\Lambda_g$. For $i\in\{u,s\}$, $\{W^i_\beta(g,z):\ z\in\Lambda_g\}$ is continuous on $g$ in the following sense: there is $\{\theta_{g,x}^i:\ x\in\Lambda\}$ where $\theta_{g,x}^i:\ W^i_\beta(f,x)\to W^i_\beta(g,h_g(x))$ is a $C^1$ diffeomorphism with $\theta_{g,x}^i(x)=h_g(x)$, such that if $g$ is $C^1-$close to $f$ then, for all $x\in\Lambda$, $\theta_{g,x}^i$ is uniformly $C^1-$close to the inclusion of $W^i_\beta(f,x)$ in $M$.

For any $\gamma\in(0, 1)$, let $\mathcal{U}_\gamma^F$ be a small $C^1$ neighborhood of $F$ (recall $F=f^N$). Taking $G\in \mathcal{U}_\gamma^F \cap \mathrm{Diff}^2(M)$ such that for every $x\in \Lambda_G$ (here $\Lambda_G$ is a locally maximal hyperbolic set for $G$), $n\in\mathbb{N}$ and $i=u, s$,
\begin{equation}\label{3star}
e^{nN(\lambda_i(\mu)-2\varepsilon)} \leq \|D_xG^{n}|_{E^i(x)}\| \leq  e^{nN(\lambda_i(\mu)+2\varepsilon)}.
\end{equation}

\begin{claim}\label{holderproperty}The following properties hold:
\begin{itemize}
  \item[(a)] $h_G|_{W^u_\beta(F, x)\cap\Lambda}$ and $\big(h_G|_{W^u_\beta(F, x)\cap\Lambda}\big)^{-1}$ are $(C_\gamma, \gamma)$-H\"{o}lder continuous for some $C_\gamma>0$.
  \item[(b)] the stable and unstable foliations
  \[\{W^s(G, z): z\in\Lambda_G\}, \quad \{W^u(G, z): z\in\Lambda_G\}\]
  are  $C^1$  and invariant for $G$. Thus the holonomy maps
      \begin{eqnarray*}
      \begin{aligned}
      \pi^s_G:\ &W^u_\beta(G,h_G(x))\cap\Lambda_G\to W^u_\beta(G,h_G(x'))\cap\Lambda_G\ \text{with}\\
                &\pi^s_G(y)=W^s_\beta(G,y)\cap W^u_\beta(G,h_G(x')),
      \end{aligned}
      \end{eqnarray*}
      and
      \begin{eqnarray*}
      \begin{aligned}
      \pi^u_G:\ &W^s_\beta(G,h_G(x))\cap\Lambda_G\to W^s_\beta(G,h_G(x''))\cap\Lambda_G\ \text{with}\\
                &\pi^u_G(z)=W^u_\beta(G,z)\cap W^s_\beta(G,h_G(x''))
      \end{aligned}
      \end{eqnarray*}
      are Lipschitz continuous.
\end{itemize}
\end{claim}

\begin{proof}
(a) See Claim $3.1$ in \cite{wwcz}.

(b) Since $G$ satisfies (\ref{3star}), we conclude
\begin{eqnarray*}
\begin{aligned}
\frac{\|DG|_{E^u}\| \cdot \|DG|_{E^s}\|}{m(DG|_{E^u})} & \leq e^{4N\varepsilon} e^{N(\lambda_s(\mu)+2\varepsilon)}\\
                                                       & = e^{N(\lambda_s(\mu)+6\varepsilon)}\\
                                                       & \leq 1,
\end{aligned}
\end{eqnarray*}
provide that $\lambda_s(\mu)+6\varepsilon <0$. By Theorem $6.3$ in \cite{HP70} we have the stable foliation is $C^1$. Similarly we obtain the unstable foliation is also $C^1$. Then the corresponding maps are uniformly $C^1$ (see \cite[pages 540--541]{psw97} for more details), this implies the desired result.

\end{proof}

We proceed to prove Lemma \ref{holderfoliation}. For any $y\in W^u_\beta(F, x)\cap\Lambda$,
\begin{eqnarray*}
\begin{aligned}
h_G(\pi^s(y))&=h_G\big( W^s_\beta(F,y)\cap W^u_\beta(F, x') \big)\\
             &=W^s_\beta(G, h_G(y)) \cap W^u_\beta(G, h_G(x'))\\
             &=\pi^s_G(h_G(y)).
\end{aligned}
\end{eqnarray*}
For the above $\gamma$, by Claim \ref{holderproperty}  there exists $D_\gamma>0$ such that
\[\pi^s=h_G^{-1} \circ \pi^s_G \circ h_G \]
is $(D_\gamma, \gamma)$-H\"{o}lder continuous. Using the same arguments one can prove $(\pi^s)^{-1}$, $\pi^u$ and $(\pi^u)^{-1}$ are also $(D_\gamma, \gamma)$-H\"{o}lder continuous.

\end{proof}

By Lemma \ref{holderfoliation} and the fact $f$ is topologically mixing on $\Lambda_\varepsilon$, one has $\dim_H(\Lambda_\varepsilon\cap W^u_\beta(f, x))$, $\underline{\dim}_B(\Lambda_\varepsilon\cap W^u_\beta(f, x))$ and $\overline{\dim}_B(\Lambda_\varepsilon\cap W^u_\beta(f,x))$ are independent of $\beta$ and $x$ (see the proof of Lemma $3.4$ in  \cite{wwcz} for more details).
Let
\[A_{\varepsilon, x}=(\Lambda_\varepsilon \cap W^u_\beta(f,x)) \times (\Lambda_\varepsilon \cap W^s_\beta(f,x))\]
be a product space. By the properties of dimension (see Theorem $6.5$ in \cite{Pes97} for details), one has
\begin{eqnarray}\label{6star}
\begin{aligned}
     & \dim_H(\Lambda_\varepsilon\cap W^u_\beta(f, x)) + \dim_H(\Lambda_\varepsilon\cap W^s_\beta(f, x))\\
\leq\ & \dim_H A_{\varepsilon, x}\\
\leq\ & \underline{\dim}_B A_{\varepsilon, x}\\
\leq\ & \overline{\dim}_B A_{\varepsilon, x}\\
\leq\ & \overline{\dim}_B (\Lambda_\varepsilon\cap W^u_\beta(f, x)) + \overline{\dim}_B (\Lambda_\varepsilon\cap W^s_\beta(f, x)).
\end{aligned}
\end{eqnarray}
Let $\Phi: A_{\varepsilon,x} \to \Lambda_\varepsilon$ be given by
\[ \Phi(y,z)=W^s_\beta(f, y) \cap W^u_\beta(f, z).\]
It is easy to see that $\Phi$ is a homeomorphism onto a neighborhood $V_x$ of $x$ in $\Lambda_\varepsilon$. For any $\gamma\in(0, 1)$, by Lemma \ref{holderfoliation} there is $E_\gamma>0$ such that $\Phi$ and $\Phi^{-1}$ are $(E_\gamma, \gamma)$-H\"{o}lder continuous (see Step $2$ in the proof of Theorem A in \cite{wwcz} for more details). By Lemma \ref{holderdimension} and the arbitrariness of $\gamma$, one has
\[ \mathrm{ Dim} V_x = \mathrm{ Dim} A_{\varepsilon, x}, \]
where $\mathrm{ Dim}$ denotes either $\dim_H$ or $\underline{\dim}_B$ or $\overline{\dim}_B$.
Since $\{V_x: x\in \Lambda_\varepsilon\}$ is an open cover of $\Lambda_\varepsilon$, one can choose a finite open cover $\{V_{x_1}, V_{x_2}, \cdots, V_{x_k}\}$ of $\Lambda_\varepsilon$. It follows from (\ref{6star}) that
\begin{eqnarray*}
\begin{aligned}
     & \dim_H(\Lambda_\varepsilon \cap W^u_\beta(f, x)) + \dim_H(\Lambda_\varepsilon \cap W^s_\beta(f, x))\\
\leq\ & \dim_H \Lambda_\varepsilon = \max_{1\leq i \leq k} \dim_H V_{x_i}\\
\leq\ & \overline{\dim}_B \Lambda_\varepsilon =\max_{1\leq i\leq k}\overline{\dim}_B V_{x_i}\\
\leq\ & \overline{\dim}_B(\Lambda_\varepsilon \cap W^u_\beta(f, x)) + \overline{\dim}_B(\Lambda_\varepsilon \cap W^s_\beta(f, x)),
\end{aligned}
\end{eqnarray*}
for every $x\in\Lambda_\varepsilon$.
This together with (\ref{4star}) and (\ref{5star}) yield that
\[\lim_{\varepsilon\to0} \mathrm{ Dim} \Lambda_\varepsilon = \frac{h_\mu(f)}{\lambda_u(\mu)} - \frac{h_\mu(f)}{\lambda_s(\mu)} = \mathrm{ Dim} \mu, \]
where $\mathrm{ Dim}$ denotes either $\dim_H$ or $\underline{\dim}_B$ or $\overline{\dim}_B$.
This completes the proof of Theorem \ref{C}.






\subsection*{Acknowledgments}
This work is partially supported by The National Key Research and Development Program of China (2022YFA1005802).
Y. Cao is partially supported by NSFC (11790274).
J. Wang is partially supported by NSFC (11501400, 12271386) and the Talent Program of Shanghai University of Engineering Science. Y. Zhao  is partially supported by NSFC (12271386) and Qinglan project of Jiangsu Province.



\end{document}